\documentclass{amsart}
\usepackage{url}

\newtheorem{thm}{Theorem}[section]
\newtheorem{cor}{Corollary}[section]

\newtheorem{lem}{Lemma}[section]

\theoremstyle{definition}

\theoremstyle{remark}
\newtheorem{rem}{Remark}[section]
\numberwithin{equation}{section}

\begin{document}

\title[Arithmetic properties of  partition quadruples  with odd parts distinct]
 {Arithmetic properties of  partition quadruples \\ with odd parts distinct}

\author{Liuquan Wang }

\address{Department of Mathematics, National University of Singapore, Singapore, 119076, SINGAPORE}

\email{wangliuquan@u.nus.edu; mathlqwang@163.com}



\dedicatory{}


\begin{abstract}
 Let $\mathrm{pod}_{-4}(n)$ denote  the number of partition quadruples of $n$ where the odd parts in each partition are distinct. We find many arithmetic properties of $\mathrm{pod}_{-4}(n)$ including the following  infinite family of congruences:  for any integers $\alpha \ge 1$ and $n \ge 0$,
 \[\mathrm{pod}_{-4}\Big({{3}^{\alpha +1}}n+\frac{5\cdot {{3}^{\alpha }}+1}{2}\Big)\equiv 0 \pmod{9}.\]
We also establish some internal congruences and some congruences modulo 2, 5 and 8 satisfied by  $\mathrm{pod}_{-4}(n)$.
\end{abstract}

\maketitle

\section{Introduction}

  Let $\mathrm{pod}_{-k}(n)$ denote the number of partition $k$-tuples of $n$ where in each partition the odd parts are distinct. For $k=1$, $\mathrm{pod}_{-1}(n)$ is often denoted as $\mathrm{pod}(n)$.
It is well-known that
  \[\psi (q)=\sum\limits_{n=0}^{\infty }{{{q}^{n(n+1)/2}}}=\frac{({{q}^{2}};{{q}^{2}})_{\infty }^{2}}{{{(q;q)}_{\infty }}}\]
  is one of Ramanujan's theta functions. Here ${{(a;q)}_{\infty }}=\prod\limits_{n=1}^{\infty }{(1-a{{q}^{n}})}$ is  standard $q$ series notation. Moreover, we introduce the notation
\[{{({{a}_{1}},{{a}_{2}},\cdots ,{{a}_{n}};q)}_{\infty }}={{({{a}_{1}};q)}_{\infty }}{{({{a}_{2}};q)}_{\infty }}\cdots {{({{a}_{n}};q)}_{\infty }}.\]
It is not difficult to find that the generating function of $\mathrm{pod}_{-k}(n)$ is
\begin{equation}\label{podk}
\sum\limits_{n=0}^{\infty }{\mathrm{pod}_{-k}(n){{q}^{n}}}=\frac{(-q;{{q}^{2}})_{\infty }^{k}}{({{q}^{2}};{{q}^{2}})_{\infty }^{k}}=\frac{1}{\psi {{(-q)}^{k}}}.
\end{equation}

In recent years, the arithmetic properties of $\mathrm{pod}_{-k}(n)$ have drawn much attention.
In 2010, Hirschhorn and Sellers \cite{hisc} studied the congruence properties of $\mathrm{pod}(n)$. They proved some infinite family of Ramanujan-type congruences including the following one:  for integers $\alpha \ge 0$ and $n\ge 0$,
\[\mathrm{pod}\Big({{3}^{2\alpha +3}}n+\frac{23\times {{3}^{2\alpha +2}}+1}{8}\Big)\equiv 0 \pmod{3}.\]
They also found some internal congruences such as
\[ \mathrm{pod}(81n+17)\equiv 5 \, \mathrm{pod}(9n+2) \pmod{27}. \]

In 2011, Radu and Sellers \cite{radu} obtained other deep congruences for $\mathrm{pod}(n)$ modulo $5$ and $7$ by using the tool of modular forms.
Chen and Xia \cite{wyc} then investigated the arithmetic properties of  $\mathrm{pod}_{-2}(n)$. They found two infinite families of  congruences modulo 3 and 5, respectively.

Recently, the author \cite{Wang} discovered many congruences modulo 7, 9 and 11 satisfied by $\mathrm{pod}_{-3}(n)$. For example, we proved that for any integers $\alpha \ge 1$ and $n\ge 0$,
\[\mathrm{pod}_{-3}\Big({{3}^{2\alpha +2}}n+\frac{23\times {{3}^{2\alpha +1}}+3}{8}\Big)\equiv 0  \pmod{9}.\]
 We continue our work by studying the congruence properties of $\mathrm{pod}_{-4}(n)$, the number of partition quadruples of $n$ where the odd parts in each partition are distinct.

The paper is organized as follows. In Section 2, we will present an infinite family of Ramanujan-type congruences: for any integers $\alpha \ge 1$ and $n\ge 0$,
\[\mathrm{pod}_{-4}\Big({{3}^{\alpha +1}}n+\frac{5\cdot {{3}^{\alpha }}+1}{2}\Big)\equiv 0 \pmod{9}.\]
In Section 3, we prove various internal congruences such as
\[\mathrm{pod}_{-4}(27n+14) \equiv -\mathrm{pod}_{-4}(9n+5) \pmod{81}.\]
In Section 4, we establish some congruences for $\mathrm{pod}_{-4}(n)$ modulo 2, 5 and 8.
For example, let $\sigma (m)$ denote the sum of all positive divisors of $m$. We establish the following arithmetic relation:
\[\mathrm{pod}_{-4}(5n+3)\equiv {{(-1)}^{n+1}}\sigma (2n+1)  \pmod {5}.\]

\section{An Infinite Family of Congruences Modulo 9}
We need some facts about $\psi (q)$ before presenting our results. The first one is the $3$-disection of $\psi (q)$
\[\psi (q)=1+q+{{q}^{3}}+{{q}^{6}}+{{q}^{10}}+{{q}^{15}}+\cdots =A({{q}^{3}})+q\psi ({{q}^{9}}).\]
where
\[A(q) =\frac{{{({{q}^{2}};{{q}^{2}})}_{\infty }}({{q}^{3}};{{q}^{3}})_{\infty }^{2}}{{{(q;q)}_{\infty }}{{({{q}^{6}};{{q}^{6}})}_{\infty }}}.\]
The proof of this result can be found in \cite{wyc,hisc}.

The following two lemmas are also of importance in our discussion.

\begin{lem}\label{basic}
Let $p$ be a prime and $\alpha $ be a positive integer. Then
\begin{displaymath}
\begin{split}
(q;q)_{\infty }^{{{p}^{\alpha }}} & \equiv ({{q}^{p}};{{q}^{p}})_{\infty }^{{{p}^{\alpha -1}}}  \pmod  {{{p}^{\alpha }}},\\
\psi (q)^{p} & \equiv \psi (q^{p}) \pmod{p}.
\end{split}
\end{displaymath}
\end{lem}
\begin{proof}
The first congruence relation is Lemma 1.2 in \cite{radu}. The second congruence relation follows from the first one and the product representation of $\psi(q)$.
\end{proof}
\begin{lem}\label{identity1}(Cf.\ \cite[Lemma 2.1]{hisc}.)
We have
\[A{{({{q}^{3}})}^{3}}+{{q}^{3}}\psi {{({{q}^{9}})}^{3}}=\frac{\psi {{({{q}^{3}})}^{4}}}{\psi ({{q}^{9}})}\]
and
\[
\frac{1}{\psi (q)}=\frac{\psi ({{q}^{9}})}{\psi {{({{q}^{3}})}^{4}}}\left(A{{({{q}^{3}})}^{2}}-qA({{q}^{3}})\psi ({{q}^{9}})+{{q}^{2}}\psi {{({{q}^{9}})}^{2}}\right).
\]
\end{lem}
Let us denote $s=A({{q}^{3}})$ and $t=\psi ({{q}^{9}})$ for convenience in this section. Then we can rewrite Lemma \ref{identity1} as
\begin{equation}\label{stone}
{{s}^{3}}+{{q}^{3}}{{t}^{3}}=\frac{\psi {{({{q}^{3}})}^{4}}}{\psi ({{q}^{9}})},
\end{equation}
\begin{equation}\label{sttwo}
\frac{1}{\psi (q)}=\frac{\psi ({{q}^{9}})}{\psi {{({{q}^{3}})}^{4}}}({{s}^{2}}-qst+{{q}^{2}}{{t}^{2}}).
\end{equation}

\begin{thm}\label{pod4mod9alpha}
For any integer $\alpha \ge 1$,  we have
\[\sum\limits_{n=0}^{\infty }{{{(-1)}^{n}}\mathrm{pod}_{-4}\Big({{3}^{\alpha }}n+\frac{{{3}^{\alpha }}+1}{2}\Big)q^n \equiv {{(-1)}^{\alpha-1 }}\frac{\psi {{({{q}^{3}})}^{4}}}{\psi {{(q)}^{8}}}} \pmod{9}.\]
\end{thm}
\begin{proof}
We proceed by induction on $\alpha$.
From (\ref{podk}) and (\ref{sttwo}), we have
\begin{equation}\label{mod272}
\sum\limits_{n=0}^{\infty }{{{(-1)}^{n}}\mathrm{pod}_{-4}(n){{q}^{n}}}= \frac{1}{\psi {{(q)}^{4}}}=\frac{\psi {{({{q}^{9}})}^{4}}}{\psi {{({{q}^{3}})}^{16}}}{({ {{s}^{2}}-qst+{{q}^{2}}{{t}^{2}})}^{4}}.
\end{equation}
If we extract all the terms of the form   ${{q}^{3k+2}}$ in the expansion of ${{({{s}^{2}}-qst+{{q}^{2}}{{t}^{2}})}^{4}}$,  we obtain
\begin{equation}\label{mod273}
{{q}^{2}}(10{{s}^{6}}{{t}^{2}}-16{{q}^{3}}{{s}^{3}}{{t}^{5}}+{{q}^{6}}{{t}^{8}})\equiv {{q}^{2}}{{t}^{2}}{{({{s}^{3}}+{{q}^{3}}{{t}^{3}})}^{2}} \pmod{9}.
\end{equation}
Hence
\begin{displaymath}
\begin{split}
  \sum\limits_{n=0}^{\infty }{{{(-1)}^{3n+2}}\mathrm{pod}_{-4}(3n+2){{q}^{3n+2}}}
 & \equiv \frac{\psi {{({{q}^{9}})}^{4}}}{\psi {{({{q}^{3}})}^{16}}}{{q}^{2}}{{t}^{2}}{{({{s}^{3}}+{{q}^{3}}{{t}^{3}})}^{2}} \\
  & ={{q}^{2}}\frac{\psi {{({{q}^{9}})}^{4}}}{\psi {{({{q}^{3}})}^{8}}} \pmod{9}.
\end{split}
\end{displaymath}
Dividing both sides by ${{q}^{2}}$, then replacing  ${{q}^{3}}$ by  $q$, we get
\[\sum\limits_{n=0}^{\infty }{{{(-1)}^{n}}\mathrm{pod}_{-4}(3n+2){{q}^{n}}}\equiv \frac{\psi {{({{q}^{3}})}^{4}}}{\psi {{(q)}^{8}}} \pmod{9}.\]
Hence the result holds when $\alpha =1$.

Suppose
\[\sum\limits_{n=0}^{\infty }{{{(-1)}^{n}}\mathrm{pod}_{-4}\Big({{3}^{\alpha }}n+\frac{{{3}^{\alpha }}+1}{2}\Big){{q}^{n}}}\equiv {{(-1)}^{\alpha -1}}\frac{\psi {{({{q}^{3}})}^{4}}}{\psi {{(q)}^{8}}} \pmod{9}.\]
Applying (\ref{sttwo}) again, we obtain
\begin{equation}\label{stpsi8}
\frac{\psi {{({{q}^{3}})}^{4}}}{\psi {{(q)}^{8}}}=\frac{\psi {{({{q}^{9}})}^{8}}}{\psi {{({{q}^{3}})}^{28}}}{({ {{s}^{2}}-qst+{{q}^{2}}{{t}^{2}})}^{8}}.
\end{equation}
If we extract all the terms of the form  ${{q}^{3k+1}}$  in the expansion of ${{({{s}^{2}}-qst+{{q}^{2}}{{t}^{2}})}^{8}}$, we obtain
\begin{displaymath}
\begin{split}
  & q(-8{{s}^{15}}t+266{{q}^{3}}{{s}^{12}}{{t}^{4}}-1016{{q}^{6}}{{s}^{9}}{{t}^{7}}+784{{q}^{9}}{{s}^{6}}{{t}^{10}}-112{{q}^{12}}{{s}^{3}}{{t}^{13}}+{{q}^{15}}{{t}^{16}}) \\
 & \equiv  qt({{s}^{15}}+5{{q}^{3}}{{s}^{12}}{{t}^{3}}+10{{q}^{6}}{{s}^{9}}{{t}^{6}}+10{{q}^{9}}{{s}^{6}}{{t}^{9}}+5{{q}^{12}}{{s}^{3}}{{t}^{12}}+{{q}^{15}}{{t}^{15}}) \\
 & \equiv qt{{({{s}^{3}}+{{q}^{3}}{{t}^{3}})}^{5}} \pmod{9}.
\end{split}
\end{displaymath}
Hence
\begin{align} \nonumber
\begin{split}
   &\sum\limits_{n=0}^{\infty }{{{(-1)}^{3n+1}}\mathrm{pod}_{-4}\Big({{3}^{\alpha }}(3n+1)+\frac{{{3}^{\alpha }}+1}{2}\Big){{q}^{3n+1}}} \\
   & \equiv {{(-1)}^{\alpha -1}}qt{{({{s}^{3}}+{{q}^{3}}{{t}^{3}})}^{5}}\frac{\psi {{({{q}^{9}})}^{8}}}{\psi {{({{q}^{3}})}^{28}}} \\ \nonumber
 & ={{(-1)}^{\alpha -1}}q\frac{\psi {{({{q}^{9}})}^{4}}}{\psi {{({{q}^{3}})}^{8}}} \pmod{9}.
\end{split}
\end{align}
Dividing both sides by $-q$  and then replacing  ${{q}^{3}}$  by  $q$,  we get
\[\sum\limits_{n=0}^{\infty }{{{(-1)}^{n}}\mathrm{pod}_{-4}\Big({3}^{\alpha +1}n+\frac{{{3}^{\alpha +1}}+1}{2}\Big){{q}^{n}}}\equiv {{(-1)}^{(\alpha +1)-1}}\frac{\psi {{({{q}^{3}})}^{4}}}{\psi {{({{q}})}^{8}}} \pmod{9}.\]
This implies that the result holds for $\alpha +1$.

By induction on $\alpha$, we complete our proof.
\end{proof}

For any positive integer $n$ and prime $p$,   we denote by ${{v}_{p}}(n)$ the power of  $p$ in the unique prime factorization  of $n$. Let $\sigma(n)$ denote the sum of positive divisors of $n$.
\begin{cor}\label{sigma}
For any integer  $n\ge 0$, we have
\[\mathrm{pod}_{-4}(3n+2)\equiv {{(-1)}^{n}}\sigma (2n+1) \pmod{3}.\]
Moreover,  $\mathrm{pod}_{-4}(3n+2)\equiv 0$ \text{\rm{(mod $3$)}}  if and only if one of the following conditions hold.\\
(1) There exists a prime  $p\equiv 1$ \text{\rm{(mod $3$)}} such that ${{v}_{p}}(2n+1)\equiv 2$ \text{\rm{(mod $3$)}}.\\
(2) There exists a prime  $p\equiv 2$ \text{\rm{(mod $3$)}} such that ${{v}_{p}}(2n+1)\equiv 1$ \text{\rm{(mod $2$)}}.
\end{cor}
\begin{proof}
Let $\alpha =1$ in Theorem \ref{pod4mod9alpha}. By Lemma \ref{basic}, we have
\[\sum\limits_{n=0}^{\infty }{{{(-1)}^{n}}\mathrm{pod}_{-4}(3n+2){{q}^{n}}}\equiv \frac{\psi {{({{q}^{3}})}^{4}}}{\psi {{(q)}^{8}}}\equiv \psi {{(q)}^{4}}=\sum\limits_{n=0}^{\infty }{{{t}_{4}}(n){{q}^{n}}} \pmod{3}.\]
Hence
\[\mathrm{pod}_{-4}(3n+2)\equiv {{(-1)}^{n}}{{t}_{4}}(n) \pmod{3}.\]
From Theorem 3.6.3 in \cite{Bruce}, we know ${{t}_{4}}(n)=\sigma (2n+1)$. Hence  we proved the first assertion.

We write the unique prime factorization of $2n+1$ as $2n+1 = \prod\nolimits_{p|2n+1}{{{p}^{{{v}_{p}}(2n+1)}}}$. Because
\[\sigma (2n+1)=\prod\nolimits_{p|2n+1}{\big(1+p+\cdots +{{p}^{{{v}_{p}}(2n+1)}}\big)},\]
we see that $3|\sigma (2n+1)$ if and only if there exists a prime $p|2n+1$ such that   $3|(1+p+\cdots +{{p}^{{{v}_{p}}(2n+1) }})$.

If $p=3$,  then  $1+p+\cdots +{{p}^{{{v}_{p}}(2n+1)}}\equiv 1$ (mod $3$).   If  $p\equiv 1$ (mod $3$),  then $1+p+\cdots +{{p}^{{{v}_{p}}(2n+1)}}\equiv 1+ {{v}_{p}}(2n+1)$ (mod $3$).  Hence $3|(1+p+\cdots +{{p}^{{{v}_{p}}(2n+1)}})$ if and only if  ${{v}_{p}}(2n+1) \equiv 2$ (mod $3$).

  If  $p\equiv 2$ (mod $3$),  then since
\[1+p+\cdots +{{p}^{{{v}_{p}}(2n+1)}}=\frac{{{p}^{{{v}_{p}}(2n+1) +1}}-1}{p-1},\]
we know that $3|1+p+\cdots +{{p}^{{{v}_{p}}(2n+1)}}$ if and only if  ${{p}^{{{v}_{p}}(2n+1)+1}}\equiv 1$ (mod $3$), i.e.,  ${{v}_{p}}(2n+1)$ is odd.
\end{proof}

\begin{thm}
For any integers $\alpha \ge 1$ and  $n\ge 0$,  we have
\[\mathrm{pod}_{-4}\Big({{3}^{\alpha +1}}n+\frac{5\cdot {{3}^{\alpha }}+1}{2}\Big)\equiv 0 \pmod{9}.\]
\end{thm}
\begin{proof}
By Lemma 2.3 in \cite{Wang}, we know the coefficient of  ${{q}^{3n+2}}$ in the series expansion of $\frac{\psi {{({{q}^{3}})}^{4}}}{\psi {{(q)}^{8}}}$ is divisible by $9$. Together with Theorem \ref{pod4mod9alpha}, we complete our proof.
\end{proof}

Note that for different integers  $\alpha \ge 1$, the arithmetic sequence $\big\{{{3}^{\alpha +1}}n+\frac{5\cdot {{3}^{\alpha }}+1}{2}:n=0,1,2,\cdots \big\}$ are disjoint, and they account for
\[\frac{1}{3^{2}}+\frac{1}{{{3}^{3}}} +\cdots +\frac{1}{{{3}^{\alpha +1}}}+\cdots =\frac{1}{6}\]
of all nonnegative integers. We get the following corollary immediately.
\begin{cor}
$\mathrm{pod}_{-4}(n)$ is divisible by $9$ for at least  $1/6$ of all nonnegative integers.
\end{cor}

\section{Some Internal Congruences}
If we let  $\alpha =1$ and $\alpha =2$ in Theorem \ref{pod4mod9alpha}, respectively, we deduce that
\[\mathrm{pod}_{-4}(9n+5)\equiv -\mathrm{pod}_{-4}(3n+2) \pmod{9}.\]
By more careful treatment, we can improve this congruence to the following theorem.
\begin{thm}\label{interpod4}
For any integer $n \ge 0$, we have
\begin{displaymath}
\begin{split}
   \mathrm{pod}_{-4}(27n+5)&\equiv -\mathrm{pod}_{-4}(9n+2) \pmod{9}, \\ \nonumber
  \mathrm{pod}_{-4}(27n+14)&\equiv -\mathrm{pod}_{-4}(9n+5) \pmod{81}, \\ \nonumber
  \mathrm{pod}_{-4}(27n+23)&\equiv -\mathrm{pod}_{-4}(9n+8) \pmod{27}.
\end{split}
\end{displaymath}
\end{thm}
Before we prove this theorem, we need to establish the following lemma first.
\begin{lem}\label{pod4mod27}
We have
\[\sum\limits_{n=0}^{\infty }{{{(-1)}^{n}}\mathrm{pod}_{-4}(3n+2){{q}^{n}}}=10\frac{\psi {{({{q}^{3}})}^{4}}}{\psi {{(q)}^{8}}}-36q\frac{\psi {{({{q}^{3}})}^{8}}}{\psi {{(q)}^{12}}}+27{{q}^{2}}\frac{\psi {{({{q}^{3}})}^{12}}}{\psi {{(q)}^{16}}}.\]
\[\sum\limits_{n=0}^{\infty }{{{(-1)}^{n}}\mathrm{pod}_{-4}(9n+5){{q}^{n}}}\equiv 35\frac{\psi {{({{q}^{3}})}^{4}}}{\psi {{(q)}^{8}}}+18q\frac{\psi {{({{q}^{3}})}^{8}}}{\psi {{(q)}^{12}}}-27{{q}^{2}}\frac{\psi {{({{q}^{3}})}^{12}}}{\psi {{(q)}^{16}}}-27q \psi {{({{q}^{3}})}^{4}} \pmod{81}.\]
\end{lem}
\begin{proof}
Let us go back to the proof of Theorem \ref{pod4mod9alpha}. By (\ref{mod272}) and (\ref{mod273}),  if we extract all the terms of the form $q^{3n+2}$  in the expansion of $\psi {{(q)}^{-4}}$, then divide by  $q^2$ and replace $q^3$ by $q$, we obtain
\[\sum\limits_{n=0}^{\infty }{{{(-1)}^{n}}\mathrm{pod}_{-4}(3n+2){{q}^{n}}}=\frac{\psi {{({{q}^{3}})}^{4}}}{\psi {{(q)}^{16}}}\left(10A{{(q)}^{6}}\psi {{({{q}^{3}})}^{2}}-16qA{{(q)}^{3}}\psi {{({{q}^{3}})}^{5}}+{{q}^{2}}\psi {{({{q}^{3}})}^{8}}\right).\]
By Lemma \ref{identity1}, we have  $A{{(q)}^{3}}\psi ({{q}^{3}})=\psi {{(q)}^{4}}-q\psi {{({{q}^{3}})}^{4}}$. Substituting this formula into the identity above, after simple manipulations, we proved the first identity.

Now we turn to the second congruence identity.
If we extract all the terms of the form ${{q}^{3k+1}}$  in the expansion of ${{({{s}^{2}}-qst+{{q}^{2}}{{t}^{2}})}^{8}}$, we get
\begin{displaymath}
\begin{split}
  & \quad -8q{{s}^{15}}t+266{{q}^{4}}{{s}^{12}}{{t}^{4}}-1016{{q}^{7}}{{s}^{9}}{{t}^{7}}+784{{q}^{10}}{{s}^{6}}{{t}^{10}}-112{{q}^{13}}{{s}^{3}}{{t}^{13}}+{{q}^{16}}{{t}^{16}} \\
 & \equiv qt{{({{s}^{3}}+{{q}^{3}}{{t}^{3}})}^{3}}(-8{{s}^{6}}-34{{s}^{3}}{{q}^{3}}{{t}^{3}}+{{q}^{6}}{{t}^{6}}) \pmod{81}.
\end{split}
\end{displaymath}

By (\ref{stpsi8}),  extracting all the terms of the form   ${{q}^{3k+1}}$  in the expansion of $\frac{\psi {{({{q}^{3}})}^{4}}}{\psi {{(q)}^{8}}}$ and doing reduction modulo $81$, we obtain
\begin{equation} \label{3k1mod81}
\begin{split}
  & q\frac{\psi {{({{q}^{9}})}^{9}}}{\psi {{({{q}^{3}})}^{28}}}\cdot \Big(\frac{\psi {{({{q}^{3}})}^{4}}}{\psi {{({{q}^{9}})}}}\Big)^{3}\Big(-8({{s}^{3}}+{{q}^{3}}{{t}^{3}})^{2}-18{{q}^{3}}{{t}^{3}}({{s}^{3}}+{{q}^{3}}{{t}^{3}})+27{{q}^{6}}{{t}^{6}}\Big) \\
 & \equiv q\frac{\psi {{({{q}^{9}})}^{6}}}{\psi {{({{q}^{3}})}^{16}}}\Big(-8\frac{\psi {{({{q}^{3}})}^{8}}}{{\psi ({{q}^{9}})}^{2}}-18{{q}^{3}}\psi {{({{q}^{9}})}^{3}}\frac{\psi {{({{q}^{3}})}^{4}}}{\psi ({{q}^{9}})}+27{{q}^{6}}\psi {{({{q}^{9}})}^{6}}\Big) \\
 & \equiv -8q\frac{\psi {{({{q}^{9}})}^{4}}}{\psi {{({{q}^{3}})}^{8}}}-18{{q}^{4}}\frac{\psi {{({{q}^{9}})}^{8}}}{\psi {{({{q}^{3}})}^{12}}}+27{{q}^{7}}\frac{\psi {{({{q}^{9}})}^{12}}}{\psi {{({{q}^{3}})}^{16}}} \pmod{81}.
\end{split}
\end{equation}
If we denote by $F(q)$ the sum of all the terms of the form ${{q}^{3k+1}}$ in the expansion of $10\frac{\psi {{({{q}^{3}})}^{4}}}{\psi {{(q)}^{8}}}$, we get
\begin{equation}\label{10term1}
\begin{split}
  F(q) & \equiv -80q\frac{\psi {{({{q}^{9}})}^{4}}}{\psi {{({{q}^{3}})}^{8}}}-180{{q}^{4}}\frac{\psi {{({{q}^{9}})}^{8}}}{\psi {{({{q}^{3}})}^{12}}}+270{{q}^{7}}\frac{\psi {{({{q}^{9}})}^{12}}}{\psi {{({{q}^{3}})}^{16}}} \\
 & \equiv q\frac{\psi {{({{q}^{9}})}^{4}}}{\psi {{({{q}^{3}})}^{8}}}-18{{q}^{4}}\frac{\psi {{({{q}^{9}})}^{8}}}{\psi {{({{q}^{3}})}^{12}}}+27{{q}^{7}}\frac{\psi {{({{q}^{9}})}^{12}}}{\psi {{({{q}^{3}})}^{16}}} \pmod{81}.
\end{split}
\end{equation}

Similarly, by (\ref{sttwo}) we have
\begin{equation}\label{lateruse}
\frac{\psi {{({{q}^{3}})}^{8}}}{\psi {{(q)}^{12}}}=\frac{\psi {{({{q}^{9}})}^{12}}}{\psi {{({{q}^{3}})}^{40}}}{{({{s}^{2}}-qst+{{q}^{2}}{{t}^{2}})}^{12}}.
\end{equation}
If we extract all the terms of the form  ${{q}^{3k}}$ in the expansion of ${{({{s}^{2}}-qst+{{q}^{2}}{{t}^{2}})}^{12}}$, we obtain
\begin{equation}\label{st12mod9}
\begin{split}
  & {{s}^{24}}-352{{s}^{21}}{{q}^{3}}{{t}^{3}}+8074{{s}^{18}}{{q}^{6}}{{t}^{6}}-43252{{s}^{15}}{{q}^{9}}{{t}^{9}}  +73789{{s}^{12}}{{q}^{12}}{{t}^{12}}-43252{{s}^{9}}{{q}^{15}}{{t}^{15}}\\
  &\quad +8074{{s}^{6}}{{q}^{18}}{{t}^{18}}-352{{s}^{3}}{{q}^{21}}{{t}^{21}}+{{q}^{24}}{{t}^{24}} \\
 & \equiv {{({{s}^{3}}+{{q}^{3}}{{t}^{3}})}^{8}} \pmod{9}.
\end{split}
\end{equation}

Now if we denote by $G(q)$ the sum of all the terms of the form  ${{q}^{3k+1}}$ in the expansion of $-36q\frac{\psi {{({{q}^{3}})}^{8}}}{\psi {{(q)}^{12}}}$, we get
\begin{equation}\label{term2}
G(q) \equiv -36q \frac{\psi {{({{q}^{9}})}^{12}}}{\psi {{({{q}^{3}})}^{40}}}{{({{s}^{3}}+{{q}^{3}}{{t}^{3}})}^{8}}
\equiv -36q\frac{\psi {{({{q}^{9}})}^{4}}}{\psi {{({{q}^{3}})}^{8}}} \pmod{81}.
\end{equation}

 In the same way,  since  $\psi {{(q)}^{3}}\equiv \psi ({{q}^{3}})$  (mod $3$), by (\ref{sttwo})  we have
\[ 27{{q}^{2}}\frac{\psi {{({{q}^{3}})}^{12}}}{\psi {{(q)}^{16}}}\equiv 27{{q}^{2}}\frac{\psi {{({{q}^{3}})}^{7}}}{\psi (q)}
  =27{{q}^{2}}\psi ({{q}^{9}})\psi {{({{q}^{3}})}^{3}}(s^2-qst+q^2t^2) \pmod{81}.\]

  If we denote by $H(q)$ the sum of all the terms of the form ${{q}^{3k+1}}$  in the expansion of $27{{q}^{2}}\frac{\psi {{({{q}^{3}})}^{12}}}{\psi {{(q)}^{16}}}$, then $H(q) \equiv 27{{q}^{4}}\psi {{({{q}^{9}})}^{3}}\psi {{({{q}^{3}})}^{3}}$ (mod 81). Together with (\ref{10term1}),  (\ref{term2})  and the first identity in this lemma, we obtain
\begin{displaymath}
\begin{split}
  & \quad \sum\limits_{n=0}^{\infty }{{{(-1)}^{3n+1}}\mathrm{pod}_{-4}\big(3(3n+1)+2\big){{q}^{3n+1}}} \\
  & =F(q)+G(q)+H(q) \\
 & \equiv -35q\frac{\psi {{({{q}^{9}})}^{4}}}{\psi {{({{q}^{3}})}^{8}}}-18{{q}^{4}}\frac{\psi {{({{q}^{9}})}^{8}}}{\psi {{({{q}^{3}})}^{12}}}+27{{q}^{7}}\frac{\psi {{({{q}^{9}})}^{12}}}{\psi {{({{q}^{3}})}^{16}}}+27{{q}^{4}}\psi {{({{q}^{9}})}^{3}}\psi {{({{q}^{3}})}^{3}} \pmod{81}.
\end{split}
\end{displaymath}
Dividing both sides by  $-q$ and replacing ${{q}^{3}}$ by  $q$,  note that $\psi {{(q)}^{3}}\equiv \psi ({{q}^{3}})$ (mod $3$), we proved the second congruence identity.
\end{proof}

Now we are ready to prove Theorem \ref{interpod4}.
\begin{proof}[Proof of Theorem  \ref{interpod4}]
By Lemma \ref{pod4mod27}, we have
\begin{equation}  \label{ker}
\begin{split}
 &\sum\limits_{n=0}^{\infty }{{{(-1)}^{n}}\big(\mathrm{pod}_{-4}(3n+2)+\mathrm{pod}_{-4}(9n+5)\big){{q}^{n}}} \\
 & \equiv 9\Big(5\frac{\psi {{({{q}^{3}})}^{4}}}{\psi {{(q)}^{8}}}-2q\frac{\psi {{({{q}^{3}})}^{8}}}{\psi {{(q)}^{12}}}-3q\psi {{({{q}^{3}})}^{4}}\Big) \pmod{81}.
\end{split}
\end{equation}
The first congruence follows immediately.

Moreover, we have
\begin{displaymath}
\begin{split}
&\quad \sum\limits_{n=0}^{\infty }{{{(-1)}^{n}}\big(\mathrm{pod}_{-4}(3n+2)+\mathrm{pod}_{-4}(9n+5)\big){{q}^{n}}} \\
 & \equiv 18\Big(\frac{\psi {{({{q}^{3}})}^{4}}}{\psi {{(q)}^{8}}}-q\frac{\psi {{({{q}^{3}})}^{8}}}{\psi {{(q)}^{12}}}\Big) \\  \nonumber
 & \equiv 18\big(\psi ({{q}^{3}})\psi (q)-q\psi {{({{q}^{3}})}^{4}}\big) \pmod{27}.
\end{split}
\end{displaymath}
Since $\psi (q)=A({{q}^{3}})+q\psi ({{q}^{9}})$,  we know the terms of the form  ${{q}^{3k+2}}$ vanish in the right hand side of the identity above. Hence
\[\mathrm{pod}_{-4}(9n+8)+\mathrm{pod}_{-4}(27n+23)\equiv 0 \pmod{27}.\]
We proved the third congruence.

  For the second congruence, we need more arguments.

 By (\ref{st12mod9}), the sum of all the terms of the form ${{q}^{3k}}$ in the expansion of ${{({{s}^{2}}-qst+{{q}^{2}}{{t}^{2}})}^{12}}$ is congruent to   ${{({{s}^{3}}+{{q}^{3}}{{t}^{3}})}^{8}}$ modulo 9. If we denote by $I(q)$ the sum of all the terms of the form  ${{q}^{3k+1}}$ in the expansion of    $-2q\frac{\psi {{({{q}^{3}})}^{18}}}{\psi {{(q)}^{12}}}$, then by (\ref{lateruse}) we have
\[I(q) \equiv -2q{{\Big(\frac{\psi {{({{q}^{3}})}^{4}}}{\psi ({{q}^{9}})}\Big)}^{8}}\frac{\psi {{({{q}^{9}})}^{12}}}{\psi {{({{q}^{3}})}^{40}}}\equiv -2q\frac{\psi {{({{q}^{9}})}^{4}}}{\psi {{({{q}^{3}})}^{8}}} \pmod{9}.\]

By (\ref{3k1mod81}), we know the sum of all the terms of the form  ${{q}^{3k+1}}$  in the expansion of $5\frac{\psi {{({{q}^{3}})}^{4}}}{\psi {{(q)}^{8}}}$ is congruent to  $-4q\frac{\psi {{({{q}^{9}})}^{4}}}{\psi {{({{q}^{3}})}^{8}}}$ modulo 9. Hence if we extract all the terms of the form ${{q}^{3k+1}}$ in both sides of  (\ref{ker}), we obtain
\begin{displaymath}
\begin{split}
  & \sum\limits_{n=0}^{\infty }{{{(-1)}^{3n+1}}\big(\mathrm{pod}_{-4}(9n+5)+\mathrm{pod}_{-4}(27n+14)\big){{q}^{3n+1}}} \\ \nonumber
 & \equiv 27\Big(-2q\frac{\psi {{({{q}^{9}})}^{4}}}{\psi {{({{q}^{3}})}^{8}}}-q\psi {{({{q}^{3}})}^{4}}\Big) \equiv -81\psi {{(q^{3})}^{4}} \pmod{81}.
\end{split}
\end{displaymath}
The second congruence follows.
\end{proof}
\begin{rem}
The modulus in Theorem  \ref{interpod4} cannot be replaced by high powers of 3. Because we have
\begin{displaymath}
\begin{split}
\mathrm{pod}_{-4}(5)+\mathrm{pod}_{-4}(2)&={{3}^{2}}\times 2\times 7, \\ \nonumber
\mathrm{pod}_{-4}(14)+\mathrm{pod}_{-4}(5)&={{3}^{4}}\times {{2}^{8}}, \\ \nonumber
\mathrm{pod}_{-4}(23)+\mathrm{pod}_{-4}(8)&={{3}^{3}}\times 19\times 2027.
\end{split}
\end{displaymath}
\end{rem}

\section{Congruences Modulo 2, 5 and 8}
\begin{thm}\label{ramaid}
We have
\[\sum\limits_{n=0}^{\infty }{\mathrm{pod}_{-4}(2n){{q}^{n}}}=\frac{({{q}^{2}};{{q}^{2}})_{\infty }^{10}}{(q;q)_{\infty }^{10}({{q}^{4}};{{q}^{4}})_{\infty }^{4}},\]
\[\sum\limits_{n=0}^{\infty }{\mathrm{pod}_{-4}(2n+1){{q}^{n}}}=4\frac{({{q}^{4}};{{q}^{4}})_{\infty }^{4}}{(q;q)_{\infty }^{6}({{q}^{2}};{{q}^{2}})_{\infty }^{2}}.\]
\end{thm}
\begin{proof}
By (\ref{podk}) we have
\[\sum\limits_{n=0}^{\infty }{\mathrm{pod}_{-4}(n){{q}^{n}}}=\frac{(-q;{{q}^{2}})_{\infty }^{4}}{({{q}^{2}};{{q}^{2}})_{\infty }^{4}}=\frac{({{q}^{2}};{{q}^{4}})_{\infty }^{4}}{(q;q)_{\infty }^{4}}=\frac{({{q}^{2}};{{q}^{2}})_{\infty }^{4}}{(q;q)_{\infty }^{4}({{q}^{4}};{{q}^{4}})_{\infty }^{4}}.\]
From \cite[Corollary 2.4]{pee} we know
\[\frac{({{q}^{2}};{{q}^{2}})_{\infty }^{14}}{(q;q)_{\infty }^{4}({{q}^{4}};{{q}^{4}})_{\infty }^{4}}=\frac{({{q}^{4}};{{q}^{4}})_{\infty }^{10}}{({{q}^{8}};{{q}^{8}})_{\infty }^{4}}+4q\frac{({{q}^{2}};{{q}^{2}})_{\infty }^{4}({{q}^{8}};{{q}^{8}})_{\infty }^{4}}{({{q}^{4}};{{q}^{4}})_{\infty }^{2}}.\]
Hence
\[\sum\limits_{n=0}^{\infty }{\mathrm{pod}_{-4}(n){{q}^{n}}}=\frac{({{q}^{2}};{{q}^{2}})_{\infty }^{4}}{(q;q)_{\infty }^{4}({{q}^{4}};{{q}^{4}})_{\infty }^{4}}=\frac{({{q}^{4}};{{q}^{4}})_{\infty }^{10}}{({{q}^{2}};{{q}^{2}})_{\infty }^{10}({{q}^{8}};{{q}^{8}})_{\infty }^{4}}+4q\frac{({{q}^{8}};{{q}^{8}})_{\infty }^{4}}{({{q}^{2}};{{q}^{2}})_{\infty }^{6}({{q}^{4}};{{q}^{4}})_{\infty }^{2}}.\]
From which the theorem follows.
\end{proof}
\begin{thm}\label{pod4mod2}
For any integer $n\ge 0$,  we have\\
(1)  $\mathrm{pod}_{-4}(4n+2)\equiv 0$  \text{\rm{(mod $2$)}}.  \\
(2)  If $n=k(k+1)$ for some integer $k$,  then $\mathrm{pod}_{-4}(2n+1)\equiv 4$ \text{\rm{(mod $8$)}}. Otherwise $\mathrm{pod}_{-4}(2n+1)\equiv 0$ \text{\rm{(mod $8$)}}.
\end{thm}
\begin{proof}
(1) By Theorem \ref{ramaid} and Lemma \ref{basic} we have
\begin{displaymath}
\sum\limits_{n=0}^{\infty }{\mathrm{pod}_{-4}(2n){{q}^{n}}}  =\frac{1}{(q;q)_{\infty }^{10}}\cdot \frac{({{q}^{2}};{{q}^{2}})_{\infty }^{10}}{({{q}^{4}};{{q}^{4}})_{\infty }^{4}} \equiv \frac{1}{({{q}^{2}};{{q}^{2}})_{\infty }^{3}} \pmod{2}.
\end{displaymath}
Note that the terms of the form ${{q}^{2n+1}}$ do not appear in the right hand side of the above equation, we deduce that  $\mathrm{pod}_{-4}(4n+2)\equiv 0$ (mod $2$).

(2)  Again by Theorem \ref{ramaid} and Lemma \ref{basic}, we have
\begin{equation}\label{pod4odd}
\sum\limits_{n=0}^{\infty }{\mathrm{pod}_{-4}(2n+1){{q}^{n}}}=\frac{4}{(q;q)_{\infty }^{6}}\cdot \frac{({{q}^{4}};{{q}^{4}})_{\infty }^{4}}{({{q}^{2}};{{q}^{2}})_{\infty }^{2}} \equiv 4({{q}^{2}};{{q}^{2}})_{\infty }^{3} \pmod{8}.
\end{equation}

By Jacobi's identity (Cf. \cite[p.14]{Bruce}),  we have
\[({{q}^{2}};{{q}^{2}})_{\infty }^{3}=\sum\limits_{k=0}^{\infty }{{{(-1)}^{k}}(2k+1){{q}^{k(k+1)}}}\equiv \sum\limits_{k=0}^{\infty }{{{q}^{k(k+1)}}} \pmod{2}.\]
Hence
\[\sum\limits_{n=0}^{\infty }{\mathrm{pod}_{-4}(2n+1){{q}^{n}}}\equiv 4\sum\limits_{k=0}^{\infty }{{{q}^{k(k+1)}}} \pmod{8}.\]
Comparing the coefficients of  ${{q}^{n}}$ on both sides, we complete our proof.
\end{proof}

\begin{thm}
Let $p\ge 3$ be a prime and $m\equiv -1$ \text{\rm{(mod $p$)}}. Suppose  $r$ is an integer such that $8r+1$ is a quadratic non-residue modulo $p$, then we have \[\mathrm{pod}_{-m}(pn+r)\equiv 0  \pmod{p}.\]
\end{thm}
\begin{proof}
By Lemma \ref{basic}, we have
\[\sum\limits_{n=0}^{\infty }{\mathrm{pod}_{-m}(n){{(-q)}^{n}}}=\frac{1}{\psi {{(q)}^{m}}}\equiv {{\Big(\frac{1}{\psi ({{q}^{p}})}\Big)}^{\frac{m+1}{p}}}\psi (q) \pmod{p}.\]
 Note that $\psi (q)=\sum\limits_{n=0}^{\infty }{{{q}^{n(n+1)/2}}}$,  and $r\equiv n(n+1)/2$ (mod $p$) is equivalent to $8r+1\equiv {{(2n+1)}^{2}}$ (mod $p$). If  $8r+1$ is a quadratic non-residue modulo $p$,  then ${{q}^{pn+r}}$ vanishes in the expansion of the right hand side. We complete our proof.
\end{proof}
In particular, let $p=5$ and $m=4$. We have
\begin{cor}\label{cormod5}
For any integer $n \ge 0$,  we have
\[\mathrm{pod}_{-4}(5n+2)\equiv \mathrm{pod}_{-4}(5n+4)\equiv 0 \pmod{5}.\]
\end{cor}

\begin{lem}\label{5disection}
We have
\[\psi (q)=A({{q}^{5}})+qB({{q}^{5}})+{{q}^{3}}\psi ({{q}^{25}}),\]
where
\[A(q)={{(-{{q}^{2}},-{{q}^{3}},{{q}^{5}};{{q}^{5}})}_{\infty }}, \, B(q)={{(-q,-{{q}^{4}},{{q}^{5}};{{q}^{5}})}_{\infty }}.\]
\end{lem}
\begin{proof}
Since $\psi (q)=\sum\limits_{n=0}^{\infty }{{{q}^{n(n+1)/2}}}$, and the residue of  $n(n+1)/2$ modulo 5 can only be  $0$, $1$ or $3$, we may assume that
\[\psi (q)=A({{q}^{5}})+qB({{q}^{5}})+{{q}^{3}}C({{q}^{5}}).\]
Because $n(n+1)/2\equiv 0$ (mod $5$) if and only if $n\equiv 0,4$ (mod $5$),  we have
\[A({{q}^{5}})=\sum\limits_{n=0}^{\infty }{{{q}^{5n(5n+1)/2}}}+\sum\limits_{n=0}^{\infty }{{{q}^{(5n+4)(5n+5)/2}}}.\]
Replacing ${{q}^{5}}$ by $q$, and replacing the index $n$ by $-n-1$ in the second summmation,  accordingly the summation range of $n$ starting form $-\infty $ and ends with $-1$,  combining the two sums together, and applying Jacobi's triple product identity, we obtain
\[A(q)=\sum\limits_{n=-\infty }^{\infty }{{{q}^{\frac{5}{2}{{n}^{2}}+\frac{1}{2}n}}}={{(-{{q}^{2}},-{{q}^{3}},{{q}^{5}};{{q}^{5}})}_{\infty }}.\]

Note that $n(n+1)/2 \equiv 1$ (mod $5$) if and only if $n\equiv 1,3$ (mod $5$),    by a similar way we can show that $B(q)={{(-q,-{{q}^{4}},{{q}^{5}};{{q}^{5}})}_{\infty }}$.

Finally,  since  $n(n+1)/2\equiv 3$ (mod  $5$) if and only if  $n\equiv 2$ (mod $5$), we have
\[{{q}^{3}}C({{q}^{5}})=\sum\limits_{n=0}^{\infty }{{{q}^{(5n+2)(5n+3)/2}}}={{q}^{3}}\psi ({{q}^{25}}).\]
The proof of this lemma is now complete.
\end{proof}

\begin{thm}
For any integer $n\ge 0$,  we have
\[\mathrm{pod}_{-4}(5n+3)\equiv {{(-1)}^{n+1}}\sigma (2n+1) \pmod{5}.\]
Moreover,  $\mathrm{pod}_{-4}(5n+3)\equiv 0$ \text{\rm{(mod $5$)}} if and only if one of the following conditions are satisfied.\\
(1) There exists prime $p\equiv 1$ \text{\rm{(mod $5$)}} such that ${{v}_{p}}(2n+1)\equiv 4$ \text{\rm{(mod $5$)}};\\
(2) There exists prime $p\equiv 2,3$ or $4$ \text{\rm{(mod $5$)}} such that ${{v}_{p}}(2n+1)\equiv 3$ \text{\rm{(mod $4$)}}.
\end{thm}
\begin{proof}
By Lemma \ref{basic}, we have $\psi {{(q)}^{5}}\equiv \psi ({{q}^{5}})$ (mod $5$).  By Lemma \ref{5disection}, we obtain
\[\sum\limits_{n=0}^{\infty }{{{(-1)}^{n}}\mathrm{pod}_{-4}(n){{q}^{n}}}=\frac{\psi (q)}{\psi {{(q)}^{5}}}\equiv \frac{1}{\psi ({{q}^{5}})}(A({{q}^{5}})+qB({{q}^{5}})+{{q}^{3}}\psi ({{q}^{25}})) \pmod{5}.\]
 If we extract all the terms of the form ${{q}^{5n+3}}$,  then divide by $-{{q}^{3}}$ and replace ${{q}^{5}}$ by $q$,   apply $\psi {{(q)}^{5}}\equiv \psi ({{q}^{5}})$ (mod $5$) again, we have
\[\sum\limits_{n=0}^{\infty }{(-1)^{n} \mathrm{pod}_{-4}(5n+3)q^n}\equiv -\frac{\psi ({{q}^{5}})}{\psi (q)}\equiv -\psi {{(q)}^{4}}=-\sum\limits_{n=0}^{\infty }{{{t}_{4}}(n){{q}^{n}}} \pmod{5}.\]
Note that ${{t}_{4}}(n)=\sigma (2n+1)$,  comparing the coefficients of  ${{q}^{n}}$ on both sides, we deduce that
\[\mathrm{pod}_{-4}(5n+3)\equiv {{(-1)}^{n+1}}\sigma (2n+1) \pmod{5}.\]

We write the unique prime factorization of $2n+1$ as $2n+1=\prod\nolimits_{p|2n+1}{{{p}^{{{v}_{p}}(2n+1)}}}$. Then
\[\sigma (2n+1)=\prod\nolimits_{p|2n+1}{\big(1+p+\cdots +{{p}^{{{v}_{p}}(2n+1)}}\big)}.\]
Let $p$ be any prime factor of  $2n+1$. If $p=5$,  then  $1+p+\cdots +{{p}^{{{v}_{p}}(2n+1)}}\equiv 1$ (mod $5$).

 If $p\equiv 1$ (mod $5$),  then $1+p+\cdots +{{p}^{{{v}_{p}}(2n+1)}}\equiv 1+{{v}_{p}}(2n+1)$ (mod $5$).  Now $5|1+p+\cdots +{{p}^{{{v}_{p}}(2n+1)}}$ if and only if  ${{v}_{p}}(2n+1)\equiv 4$ (mod $5$).

If  $p\equiv 2,3$ or $4$ (mod $5$), then
\[1+p+\cdots +{{p}^{{{v}_{p}}(2n+1)}}=\frac{{{p}^{{{v}_{p}}(2n+1)+1}}-1}{p-1} \pmod{5}.\]
Hence $5|1+p+\cdots +{{p}^{{{v}_{p}}(2n+1)}}$ if and only if  ${{p}^{{{v}_{p}}(2n+1)+1}}\equiv 1$  (mod $5$),  which is also equivalent to ${{v}_{p}}(2n+1)\equiv 3$ (mod $4$).
\end{proof}

\bibliographystyle{amsplain}

\end{document}